\documentclass[11pt]{article}
\usepackage{amsmath, amssymb, amsfonts, amsthm, graphicx}

\newtheorem{theorem}{Theorem}

\newtheorem{thm}[theorem]{Theorem}

\newcommand\FP{\Phi} 

\title{Median eigenvalues of bipartite graphs}

\author{%
  Bojan Mohar\thanks{Supported in part by an NSERC Discovery Grant (Canada),
    by the Canada Research Chair program, and by the
    Research Grant P1--0297 of ARRS (Slovenia).}~\thanks{On leave from:
    IMFM \& FMF, Department of Mathematics, University of Ljubljana, Ljubljana,
    Slovenia.}\\[2mm]
  Department of Mathematics\\
  Simon Fraser University\\
  Burnaby, BC, Canada\\
  {\tt mohar@sfu.ca}
\and
  Behruz Tayfeh-Rezaie\\[2mm]
  School of Mathematics\\
  Institute for Research in Fundamental Sciences (IPM)\\
  P.O. Box 19395-5746, Tehran, Iran\\
  {\tt tayfeh-r@ipm.ir}
}

\date{}

\begin{document}

\maketitle

\begin{abstract}
For a graph $G$ of order $n$ and with eigenvalues $\lambda_1\geqslant\cdots\geqslant\lambda_n$,
the HL-index $R(G)$ is defined as
$R(G) ={\max}\left\{|\lambda_{\lfloor(n+1)/2\rfloor}|, |\lambda_{\lceil(n+1)/2\rceil}|\right\}.$
We show that for every connected bipartite graph $G$ with maximum degree $\Delta\geqslant3$, $R(G)\leqslant\sqrt{\Delta-2}$ unless $G$ is the the incidence graph of a projective plane of order $\Delta-1$. We also present an approach through graph covering to construct infinite families of bipartite graphs with large HL-index.

\vspace{5mm}
\noindent {\bf AMS Mathematics Subject Classification\,(2010):} 05C50.
\end{abstract}

\section{Introduction}

Unless explicitly stated, we assume that all graphs in this paper are simple, i.e. multiple edges and loops
are not allowed.
The \emph{adjacency matrix} of $G$, denoted by $A(G)=(a_{uv})_{u,v\in V(G)}$, is a $(0,1)$-matrix whose
rows and columns are indexed by the vertices of $G$ such that $a_{uv}=1$ if and only if $u$ is adjacent to $v$.
We use $\deg_G(v)$ to denote the degree of vertex $v$ in $G$. The set of all neighbours of $v$ is denoted by $N_G(v)$ and we write $N_G[v]=N_G(v)\cup \{v\}$.
The smallest and largest degrees of $G$ are denoted by $\delta(G)$ and $\Delta(G)$, respectively.

Let $\lambda_1\geqslant\cdots\geqslant\lambda_n$ be the eigenvalues of $G$. Then $\lambda_{\lfloor(n+1)/2\rfloor}$ and $\lambda_{\lceil(n+1)/2\rceil}$ are
called the \emph{ median } eigenvalue(s) of $G$. These eigenvalues play an important role in mathematical chemistry since they are related to the HOMO-LUMO separation, see, e.g. \cite{GuPol} and \cite{FP1,FP2}.
Following \cite{FJP}, we define the \emph{HL-index} $R(G)$ of the graph $G$ as
$$R(G) = \max\left\{|\lambda_{\lfloor(n+1)/2\rfloor}|, |\lambda_{\lceil(n+1)/2\rceil}|\right\}.$$

If $G$ is a bipartite graph, then $R(G)$ is equal to $\lambda_{n/2}$ if $n$ is even and $0$, otherwise.
In this paper, we show that for every connected bipartite graph $G$ with maximum degree $\Delta$, $R(G)\leqslant\sqrt{\Delta-2}$ unless $G$ is the the incidence graph of a projective plane of order $\Delta-1$, in which case it is equal to $\sqrt{\Delta-1}$.
This extends the result of one of the authors \cite{Mo2} who proved the same for subcubic graphs.

On the other hand, we present an approach through graph covering to construct infinite families of connected graphs with large HL-index. Graph coverings and analysis of their eigenvalues were instrumental in a recent breakthrough in spectral graph theory by Marcus, Spielman, and Srivastava who used graph coverings to construct infinite families of Ramanujan graphs of arbitrary degrees \cite{MSS1} (and for solving the Kadison-Singer Conjecture \cite{MSS2}). In our paper, we find another application of a different character. As opposed to double covers used in \cite{MSS1}, we use $k$-fold covering graphs with cyclic permutation representation and show that the behavior of median eigenvalues can be controlled in certain instances. The main ingredient is a generalization of a result of Bilu and Linial \cite{BL} that eigenvalues of double covers over a graph $G$ are the union of the eigenvalues of $G$ and the eigenvalues of certain cover matrix $A^-$ that is obtained from the adjacency matrix by replacing some of its entries by $-1$. In our case, we use a family $A^\lambda$ of such matrices, where instead of $-1$ we use certain powers of a parameter $\lambda\in [-1,1]$. This result seems to be of independent interest.

\section{Bounds for bipartite graphs}

In this section we obtain upper bounds on the HL-index of bipartite graphs in terms of maximum and minimum degrees of graphs.
We consider regular graphs first.

\begin{thm}\label{regmed}
Let $G$ be a connected bipartite $k$-regular graph, where $k\ge3$. If\/ $R(G)>\sqrt{k-2}$, then
$R(G)=\sqrt{k-1}$ and $G$ is the incidence graph of a projective plane of order $k-1$.
\end{thm}

\begin{proof}
Let $|V(G)|=2n$. The adjacency matrix of $G$ can be written as
$$A(G)=\displaystyle{\begin{bmatrix} 0 & B\\ B^T & 0 \end{bmatrix}},$$
where $B$ is a square matrix of order $n$. The matrix $E=BB^T-kI$ is a symmetric matrix of order $n$. Assuming that $R(G)>\sqrt{k-2}$, every eigenvalue $\lambda$ of $G$ satisfies $\lambda^2 > k-2$ and hence all eigenvalues of $E$ are greater than $-2$. Hence, $E+2I$ is a positive definite matrix. All diagonal entries of this matrix are equal to $2$. Positive definiteness in turn implies that all off diagonal entries are $0$ or $1$.
It follows that $E$ is the adjacency matrix of a graph $H$ with the least eigenvalue greater than $-2$.
We see that $H$ is regular since $E\,\mathbf{j}=(BB^T-kI)\,\mathbf{j}=(k^2-k)\,\mathbf{j}$.
The connectedness of $G$ also yields that $H$ is connected. By Corollary 2.3.22 of \cite{CRS},
a connected regular graph with least eigenvalue
greater than $-2$ is either a complete graph or an odd cycle. If $H$ is an odd cycle,
then it is $2$-regular and so from $k^2-k=2$, we have $k=2$, a contradiction. Hence $H$ is a
complete graph. It is easy to see that this implies that $G$ is the incidence graph of a projective plane of order $k-1$.
\end{proof}

For the next theorem, we need the following result \cite[Theorem 2.3.20]{CRS}.

\begin{thm}[\cite{CRS}] \label{thmminge2}
 If $G$ is a connected graph with the least eigenvalue
greater than $-2$, then one of the following holds:
\begin{itemize}
 \setlength{\itemsep}{0pt}
 \setlength{\parskip}{0pt}
 \setlength{\parsep}{0pt}
\item[\rm (i)] $G$ is the line graph of a multigraph $K$, where $K$ is obtained from a tree by adding one edge in parallel to a pendant edge;
\item[\rm (ii)] $G$ is the line graph of a graph $K$, where $K$ is a tree or is obtained from a tree by adding one edge giving a nonbipartite unicyclic graph;
\item[\rm (iii)] $G$ is one of the $573$ exceptional graphs on at most $8$ vertices.
\end{itemize}
\end{thm}

We can now prove an analogue to Theorem \ref{regmed} for non-regular graphs.

\begin{thm}\label{deltaplusmed}
Let $G$ be a connected bipartite nonregular graph with maximum degree $\Delta\geqslant3$. Then $R(G)\leqslant\sqrt{\Delta-2}$.
\end{thm}

\begin{proof}
Let $d=\Delta-1$. Suppose, for a contradiction, that $R(G)>\sqrt{d-1}$. Let $\{U,W\}$ be the bipartition of $V(G)$. Then
$U$ and $W$ have the same size, say $m$, since otherwise $R(G)$ would be zero. We proceed in the same way as in the proof of Theorem \ref{regmed}. The adjacency matrix of $G$ can be written in the form
$$A(G)=\displaystyle{\begin{bmatrix} 0 & B\\ B^T & 0 \end{bmatrix}},$$
where the rows of $B$ are indexed by the elements of $U$ and the columns by $W$.
The matrix $E=BB^T-(d-1)I$ is a symmetric matrix of order $m$. Since $R(G)>\sqrt{d-1}$, we have $\lambda_m(E)>0$. Hence $E$ is a positive definite matrix whose diagonal entries are the integers $\deg_G(u)-(d-1)\leqslant2, u\in U$. Since $E$ is positive definite, these are all equal to $1$ or $2$ and hence the degrees of vertices in $U$ are either $\Delta$ or $\Delta-1$.
Moreover, this in turn implies that all off-diagonal entries of $E$ are either $0$ or $1$. Since the off-diagonal entries in $E$ count the number of walks of length $2$ between vertices in $U$, the last conclusion in particular implies that $G$ has no $4$-cycles.
Let $D$ be the diagonal matrix whose diagonal is the same as the main diagonal of $E$.
Let $H$ be the graph on $U$ with the adjacency matrix $A(H)=E-D$.
Then the least eigenvalue of $H$ is greater than $-2$ and $A(H)+D$ is positive definite. The connectedness of $G$ yields that $H$ is connected.

Suppose that $v_1,v_2\in U$ are distinct vertices of degree $d$ in $G$. Let $P$ be a shortest path in $H$ connecting $v_1$ to $v_2$. The vertices $v_1,v_2$ and the path $P$ can be chosen so that all internal vertices on $P$ are of degree $d+1$ in $G$.
Then $A(P)+{\text{diag}}(1,2,\ldots,2,1)$,
which is a principal submatrix of $A(H)+D$, has eigenvalue $0$ with the eigenvector
$(1,-1,1,-1,\ldots)^T$, a contradiction.
This shows that $U$ contains at most one vertex which has degree $d$ in $G$.
Note that the same argument can be applied to $W$. As $G$ is not regular, we conclude that it has precisely two vertices of degree $d$, one in $U$ and one in $W$. Therefore, we may assume that $D={\mathrm{diag}}(2,\ldots,2,1)$.

Let us now consider degrees of vertices in $H$. Since $G$ has no $4$-cycles, we have for every $u\in U$:
\begin{equation}\label{eq1}
\deg_H(u)=\sum_{uv\in E(G)}(\deg_G(v)-1).
\end{equation}
Since $U$ and $W$ each has precisely one vertex whose degree in $G$ is $d$, (\ref{eq1}) implies the following: If $\deg_G(u)=d$, then $\deg_H(u)\in\{d^2,d^2-1\}$; if $\deg_G(u)=d+1$, then $\deg_H(u)\in\{d^2+d,d^2+d-1\}$ with the smaller value only when $u$ is adjacent to the vertex of degree $d$ in $W$. Thus, $H$ has a unique vertex of degree $d^2$ or $d^2-1$, at most $d$ vertices of degree $d^2+d-1$, and all other vertices are of degree $d^2+d$.

Let $v\in U$ be the vertex with $\deg_G(v)=d$. We claim that
the neighbourhood of $v$ in $H$ is a complete graph. This is
since $$\displaystyle{\begin{bmatrix} 2 & 0 & 1 \\ 0 & 2 & 1\\ 1 & 1 & 1 \end{bmatrix}}$$
is not positive definite and so it cannot be a principal submatrix of $A(H)+D$.
We also claim that $N_H[v_1]\cap N_H[v_2]\subseteq N_H[v]$ for arbitrary distinct vertices $v_1, v_2\in N_H(v)$. If this were not the case, the matrix $$\displaystyle{\begin{bmatrix} 2 & 1 & 1 & 0 \\ 1 & 2 & 1 & 1 \\ 1 & 1 & 2& 1\\
0 & 1 & 1& 1 \end{bmatrix}}$$
which is not positive definite, would be a principal submatrix of $A(H)+D$, a contradiction.
Therefore, $H$ has at least $1+(d^2-1)+(d^2-1)((d^2+d-1)-(d^2-1))=d^3+d^2-d$ or
$1+d^2+d^2((d^2+d-1)-d^2)=d^3+1$ vertices. In any case, we have $|V(H)|\geqslant 9$ and so
 $H$ is the case (i) or (ii) of Theorem \ref{thmminge2}.

Let $H$ be the line graph of a multigraph $K$ as in Theorem \ref{thmminge2},
where $|V(K)|=m$ or $m+1$ and $|E(K)|=m$. Note that $K$ is not a cycle, since then $H$ would be regular in that case.
So, from $\sum_{v\in V(K)}\deg_K(v)=2m$ and $|V(K)|\geqslant m$, we have a vertex $u$ such that $\deg_K(u)=1$. Let $u'$ be the unique neighbour of $u$ in $K$. The degree of the vertex in $H$ corresponding to the edge $uu'$ is at least $d^2-1$. Thus, $\deg_K(u)+\deg_K(u')-2\geqslant d^2-1$ and so $\deg_K(u')\geqslant d^2$.
Let $r$ be the number of vertices of degree one in $K$. The sum of degrees in $K$ is at least $r+d^2+2(|V(K)|-r-1)\leqslant 2m$
which gives $r\geqslant d^2-2\geqslant2$. Since $H$ has only one vertex of degree $d^2-1$ or $d^2$,
we may take $u$ to be a vertex of degree 1 in $K$ such that $\deg_K(u)+\deg_K(u')-2\geqslant d^2+d-1$
and now a similar argument as above gives $r\geqslant d^2+d-2>d+1$. Now again since $H$ has at most $d+1$ vertices of degree $d^2-1$, $d^2$ or $d^2+d-1$,
we may take $u$ such that $\deg_K(u)+\deg_K(u')-2=d^2+d$.
It follows that $H$ has the complete graph of order $d^2+d+1$ as a subgraph
which in turn implies that $H$ is in fact the complete graph of order $d^2+d+1$ since $H$ is connected with maximum degree $d^2+d$. And this contradicts the fact that $H$ is nonregular.
\end{proof}

Theorems \ref{regmed} and \ref{deltaplusmed} can be combined into our main result.

\begin{thm}\label{mainres}
Let $G$ be a connected bipartite graph with maximum degree $\Delta\geqslant3$. Then $R(G)\leqslant\sqrt{\Delta-2}$ unless $G$ is the the incidence graph of a projective plane of order $\Delta-1$, in which case $R(G)=\sqrt{\Delta-1}$.
\end{thm}

There are connected graphs that are not incidence graphs of a projective planes and attain the bound $\sqrt{\Delta-2}$ of Theorem \ref{mainres}. The incidence graph of a biplane ($(v,k,2)$ symmetric design) has degree $k$ and HL-index $\sqrt{k-2}$. Only 17 biplanes are known and the question of existence of infinitely many biplanes is an old open problem in design theory \cite{Hall}. An infinite family of cubic graphs with HL-index equal to 1 is constructed in \cite{GM}.

Finally, we give an upper bound for the HL-index of bipartite graphs in term of the minimum degree.

\begin{thm}\label{deltamed}
Let $G$ be a bipartite graph with minimum degree $\delta$. Then
$R(G)\leqslant \sqrt{\delta}$. Moreover, if $\delta\ge2$ or $\delta=1$ and $G$ has a component
with minimum degree $1$ that is not isomorphic to $K_2$, then $R(G)<\sqrt{\delta}$.
\end{thm}

\begin{proof}
We may assume that $G$ is connected and of even order $n=2m$. Let $v$ be a vertex of degree
$\delta$ and $H=G-v$. Since $G$ is connected, $\lambda_1(G)>\lambda_1(H)$. By interlacing,
$\lambda_i(G)\geqslant\lambda_i(H)$ for $1<i\le m$. We also have $\lambda_m(H)=0$ since $H$ is bipartite and has an odd number of vertices.

The sum of the squares of the eigenvalues of $G$ is the trace of $A^2$, which is equal to $2|E(G)|$. By considering only half of the eigenvalues and using the fact that eigenvalues of a bipartite graph are symmetric about zero, we have:
\begin{eqnarray*}
\lambda_m^2(G)&=&|E(G)|-\sum_{i=1}^{m-1} \lambda_i^2(G)\\
&=&\delta+\sum_{i=1}^{m-1} \lambda_i^2(H)-\sum_{i=1}^{m-1} \lambda_i^2(G)\\[2mm]
&\leqslant& \delta.
\end{eqnarray*}
If $m\geqslant2$, then for $i=1$, we have $\lambda_1^2(H)- \lambda_1^2(G)<0$, so
the last inequality is strict. This proves the assertion of the theorem.
\end{proof}

\section{Covering graphs and their eigenvalues}

If $\hat G$ is a covering graph of $G$, then all eigenvalues of $G$ are included in the spectrum of $\hat G$. The essence of this section is to show how to control the newly arising eigenvalues in the covering graph.

We will denote the eigenvalues of a symmetric $n\times n$ matrix $M$ by
$\lambda_1(M)\geqslant\cdots\geqslant \lambda_n(M)$. Also,
if $\mathbf{x}$ is an eigenvector of $M$, then we denote the corresponding eigenvalue by $\lambda_\mathbf{x}(M)$.
For a positive integer $t$, let $I_t$ and $\mathbf{0}_{t}$ denote the $t\times t$ identity matrix and the $t\times t$ all-zero matrix, respectively.
A \emph{permutation matrix} $C$ of \emph{size} $t$ and \emph{order} $m$ is a $t\times t$
$(0,1)$-matrix that has exactly one entry $1$ in each row and each column and $m$ is the smallest positive integer such that $C^m=I_t$.

Let us replace each edge of a multigraph $G$ by two oppositely oriented directed edges joining the same pair of vertices and let $\overrightarrow{E}(G)$ denote the resulting set of
directed edges. We denote by $(e,u,v)\in \overrightarrow{E}(G)$ the
directed edge in $\overrightarrow{E}(G)$ corresponding to an edge $e=uv$ that is oriented from $u$ to $v$. Let $S_t$ denote the symmetric group of all permutations of size $t$. We shall consider a representation of $S_t$ as the set of all permutation matrices of size $t$.
A function $\phi: \overrightarrow{E}(G)\rightarrow S_t$ is a \emph{permutation voltage assignment} for $G$ if $\phi(e,u,v)=\phi(e,v,u)^{-1}$ for every $e\in E(G)$.
A \emph{$t$-lift} of $G$ \emph{associated to} $\phi$ and denoted by $G(\phi)$, is a multigraph with the adjacency matrix obtained from the adjacency matrix of $G$ by
replacing any $(u,v)$-entry of $A(G)$ by the $t\times t$ matrix $\sum_{(e,u,v)\in \overrightarrow{E}(G)} \phi(e,u,v)$.
Note that if $G$ is bipartite, then so is $G(\phi)$.
We say that $G(\phi)$ is an \emph{Abelian lift} if all matrices in the image of $\phi$ commute
with each other.

Bilu and Linial \cite{BL} found an expression for the spectrum of $2$-lifts. They proved that the spectrum of a 2-lift $G(\phi)$ consists of the spectrum of $G$ together with the spectrum of the matrix $A^-$ which is obtained from the adjacency matrix of $G$ by replacing each $(u,v)$-entry by $-1$ whenever the voltage $\phi(e,u,v)$ is not the identity.
Note that $2$-lifts are always Abelian since the permutation matrices of size 2 commute with each other.
Below we extend the result of \cite{BL} to arbitrary Abelian $t$-lifts. Since permutation matrices are diagonalizable and any commuting family of diagonalizable $t\times t$ matrices has a common basis of eigenvectors, we observe that any commuting set of
permutation matrices of the same size has a common basis of eigenvectors.

In the proofs we will use the following result, see \cite[Theorem 1]{KOV}.

\begin{thm}[\cite{KOV}]
\label{kov}
Let $t$ and $n$ be positive integers and for $i, j\in\{1, \ldots, n\}$, let $B_{ij}$ be $t\times t$ matrices over a commutative ring that commute pairwise. Then
$$\displaystyle{\det\begin{bmatrix} B_{11} & \cdots & B_{1n}\\ \vdots & \ddots & \vdots \\ B_{n1} & \cdots & B_{nn} \end{bmatrix}=\det\biggl(\,\sum_{\sigma\in S_n}\mathrm{sign}(\sigma)B_{1\sigma(1)}\cdots B_{n\sigma(n)}\biggr)},$$
where $S_n$ is the set of all permutations of\/ $\{1, \ldots, n\}$.
\end{thm}

\begin{thm}\label{eigenlift}
Let $G$ be a multigraph and $\phi$ be a permutation voltage assignment for an Abelian $t$-lift $G(\phi)$ of $G$. Let $\mathcal{B}$ be a common basis of eigenvectors of the permutation matrices in the image of\/ $\phi$.
For every $\mathbf{x}\in \mathcal{B}$, let $A_{\mathbf{x}}$ be the matrix obtained from the adjacency matrix of $G$ by replacing any $(u,v)$-entry of $A(G)$ by\/ $\sum_{(e,u,v)\in\overrightarrow{E}(G)} \lambda_\mathbf{x}(\phi(e,u,v))$.
Then the spectrum of $G(\phi)$ is the multiset union of the spectra of the matrices $A_{\mathbf{x}}$ $(\mathbf{x}\in \mathcal{B})$.
\end{thm}

\begin{proof}
The adjacency matrix of a $t$-lift can be written in the block form, with the blocks being indexed by $V(G)$, where the $(u,v)$-block $D_{uv}$ is equal to the permutation matrix $\phi(e,u,v)$ if $u,v$ are joined by a single edge $e$, or to $\sum_{(e,u,v)\in\overrightarrow{E}(G)} \phi(e,u,v)$ if there are multiple edges, or $\mathbf{0}_{t}$ if $u,v$ are not adjacent in $G$. Thus, assuming $V(G)=\{1,\ldots,n\}$, we can write
$$\lambda I-A(G(\phi))=\displaystyle{\begin{bmatrix} B_{11} & \cdots & B_{1n}\\ \vdots & \ddots & \vdots \\ B_{n1} & \cdots & B_{nn} \end{bmatrix}},$$
where the diagonal blocks are $\lambda I_t$, while the off-diagonal blocks are $B_{uv}=-D_{uv}$. All block matrices $B_{uv}$ commute with each other and all their products and sums also commute and have $\mathcal{B}$ as a common basis of eigenvectors.
By Theorem \ref{kov}, we have
\begin{align*}
\det(\lambda I-A(G(\phi)))&=\det\biggl(\,\sum_{\sigma\in S_n}\mathrm{sign}(\sigma)B_{1\sigma(1)}\cdots B_{n\sigma(n)}\biggr)\\
&=\prod_{\mathbf{x}\in \mathcal{B}}\lambda_\mathbf{x}\biggl(\,\sum_{\sigma\in S_n}\mathrm{sign}(\sigma)B_{1\sigma(1)}\cdots B_{n\sigma(n)}\biggr)\\
&=\prod_{\mathbf{x}\in \mathcal{B}}\biggl(\,\sum_{\sigma\in S_n}\mathrm{sign}(\sigma)\lambda_\mathbf{x}(B_{1\sigma(1)})\cdots \lambda_\mathbf{x}(B_{n\sigma(n)})\biggr)\\[2mm]
&=\prod_{\mathbf{x}\in \mathcal{B}} \det(\lambda I-A_\mathbf{x}).
\end{align*}
This equality gives the conclusion of the theorem.
\end{proof}

Let us keep the notation of Theorem \ref{eigenlift} and its proof. There are two things to be observed. Since every permutation matrix $D_{uv}$ satisfies $D_{uv}^{-1}=D_{uv}^T$, we have that the matrices $A_\mathbf{x}$ ($\mathbf{x}\in \mathcal{B}$) are Hermitian. The $(u,v)$-entry of $A_\mathbf{x}$ is the eigenvalue of $D_{uv}$ corresponding to the eigenvector $\mathbf{x}$. Thus the characteristic polynomial $\varphi(A_\mathbf{x},\lambda) = \det(\lambda I-A_\mathbf{x})$ is a polynomial in $\lambda$, whose coefficients are polynomials in the $\mathbf{x}$-eigenvalues of the permutation matrices in the image of $\phi$.

\section{Median eigenvalues of covering graphs}

In this section we present an approach through graph covering to construct infinite families of graphs with large HL-index. The main tool is Theorem \ref{eigenlift}, which will be invoked in a very special situation.

Let us select a set $F\subseteq \overrightarrow{E}(G)$ of oriented edges such that whenever $(e,u,v)\in F$, the opposite edge $(e,v,u)$ is not in $F$. For every positive integer $t$, let $C_t$ be a cyclic permutation of size and order $t$. Now, let us consider an infinite family of Abelian lifts $\phi_1,\phi_2,\phi_3,\dots$ such that $\phi_t$ is an Abelian $t$-lift over the graph $G$, whose voltages are given by the following rule:
\begin{equation}
\label{eq:define phit}
   \phi_t(e,u,v) = \left\{
                     \begin{array}{ll}
                       C_t, & (e,u,v)\in F; \\
                       C_t^{-1}, & (e,v,u)\in F; \\
                       I_t, & \hbox{otherwise.}
                     \end{array}
                   \right.
\end{equation}
In this way, we obtain an infinite family of graphs $G(\phi_t)$.
By Theorem \ref{eigenlift}, we can express the characteristic polynomial of $A(G(\phi_t))$ as a
product of the characteristic polynomials of matrices $A_{\mathbf{x}}$. For each $\mathbf{x}\in \mathcal{B}$, the characteristic polynomial of $A_{\mathbf{x}}$ depends only on $\lambda$ and on the eigenvalue $\alpha=\lambda_\mathbf{x}(C_t)$ and on
$\lambda_\mathbf{x}(C_t^{-1}) = \alpha^{-1} = \overline{\alpha}$. The dependence on $\alpha$ can be expressed in terms of the real parameter $\nu=\alpha+\overline{\alpha}$. For cyclic permutations, every such $\nu$ is an eigenvalue of the $t$-cycle, which is of the form $\nu=2\cos(2\pi j/t)$ for some $j\in\{0,1,\dots,t-1\}$.
Thus, there is a polynomial $\FP(\lambda,\nu)$ such that
\begin{equation}
\label{eq:polynomial}
\det(\lambda I-A(G(\phi_t)) = \prod_{0\leqslant j < t} \FP(\lambda,2\cos(2\pi j/t)).
\end{equation}
Note that $\FP(\lambda,\nu)$ is independent of $t$ and only depends on the underlying graph $G$ and the values for $\nu$ lie in the interval $[-2,2]$ for every $t$. All eigenvalues of $G(\phi_t)$ correspond to the zero-set of the polynomial $\FP(\lambda,\nu)$ with $\nu\in [-2,2]$. When $t$ gets large, the appropriate values of $\nu$ become dense in the interval $[-2,2]$.
This shows that if $G$ is bipartite, then
$R(G(\phi_t))$ converges to some value when $t$ goes to infinity. This
is better seen in a special case which is given in the following theorem.

\begin{thm}\label{liftoneedge}
Let $G$ be a bipartite graph and let $E_0$ be a set of edges all incident with some fixed vertex $v_0$. Let $F\subset \overrightarrow{E}(G)$ be the set of directed edges $\{(e,u,v_0)\mid uv_0\in E_0\}$.
For each positive integer $t$, fix a cyclic permutation matrix $C_t$ of size and order $t$ and define a permutation voltage assignment $\phi_t$ by {\rm (\ref{eq:define phit})}. Then 
$$R(G(\phi_{2t}))=R(G(\phi_2))$$
for every $t\ge1$, whereas the values $R(G(\phi_{2t+1}))$ are non-increasing as a function of\/ $t$ and $$\lim_{t\to\infty}R(G(\phi_{2t+1}))=R(G(\phi_2)).$$
\end{thm}

\begin{proof}
By Theorem \ref{eigenlift}, there is a polynomial $\FP(\lambda,\nu)$ such that (\ref{eq:polynomial}) holds and every eigenvalue of any $G(\phi_t)$ lies among the values $\lambda$ for which there is a $\nu\in[-2,2]$ such that $\FP(\lambda,\nu)=0$.
It is easy to see that our choice of $F$ implies that $\FP(\lambda,\nu)$ is linear in $\nu$, so it can be expressed in the form
$$
   \FP(\lambda,\nu)=p(\lambda)-\nu q(\lambda).
$$
If $R(G(\phi_2))$ is zero, then $R(G(\phi_{2t}))=0$ for every $t\ge1$, since by Theorem \ref{eigenlift}, the spectrum of $G(\phi_2)$ is contained in the spectrum of $G(\phi_{2t})$. Hence, we may assume that $R(G(\phi_2))\neq 0$. Note that $R(G)\ge R(G(\phi_2))$, so we also have $R(G)\ne0$.

Let us first assume that $q(0)\neq 0$. Let $\FP(0,\nu_0)=0$. Then $\nu_0=p(0)/q(0)$. We have
$\FP(0,2)=p(0)-2q(0)$ and $\FP(0,-2)=p(0)+2q(0)$ which results in
$$
   \nu_0 = p(0)/q(0) = \frac{2(\FP(0,-2)+\FP(0,2))}{\FP(0,-2)-\FP(0,2)}.
$$
On the other hand, Eq.(\ref{eq:polynomial}) gives that $\FP(0,2)=\det(-A(G))$ and $\FP(0,-2)=\det(-A(G(\phi_2)))/\det(-A(G))$. Since the eigenvalues of bipartite graphs $G$ and $G(\phi_2)$ are symmetric about zero, this implies that the above determinants, and thus also $\FP(0,2)$ and $\FP(0,-2)$, have the same sign. It follows that $|\nu_0|>2$. 
Since $\FP(\lambda,\nu)$ is linear in $\nu$, for each $\lambda$ there exists at most one value $\nu$  such that $\FP(\lambda,\nu)=0$ (and there is exactly one if $q(\lambda)\ne0$).
Therefore, the continuity of $\FP(\lambda,\nu)$ and its linearity in $\nu$ show that the eigenvalue $R(G(\phi_t))$ is either a root of $\FP(\lambda,2)$ or a root of
$\FP(\lambda,-2)$ (if $t$ is even) or a root of $\FP(\lambda,2\cos(\pi (t-1)/t))$ (if $t$ is odd).
This is independent of $t$ when $t$ is even and is already among the eigenvalues of $G(\phi_2)$. For odd values of $t$, this shows the behavior as claimed in the theorem.

Suppose next that $q(0)=0.$ Then $p(0)\neq 0$, since otherwise we have $\FP(0,2)=0$ and so $R(G)=0$, a contradiction. This shows that if $\FP(\lambda_0,\nu_0)=0$ and $\lambda_0$ goes to zero, then $\nu_0$ goes to infinity.
Again the continuity of $\FP(\lambda,\nu)$ and its linearity in $\nu$ show that $R(G(\phi_t))$ is either a root of
$\FP(\lambda,2)$ or a root of $\FP(\lambda,-2)$ (if $t$ is even) or a root of $\FP(\lambda,2\cos(\pi (t-1)/t))$ (if $t$ is odd).
We now complete the proof in the same was as above.
\end{proof}

\begin{thm}\label{infinit}
For any integer $k$ for which $k-1$ is a prime power, there exist infinitely many connected bipartite $k$-regular graphs $G$ with $\sqrt{k-2}-1<R(G)<\sqrt{k-1}-1$.
\end{thm}

\begin{proof}
Let $G$ be the incidence graph of a projective plane of order
$q=k-1$. Note that $G$ is bipartite and $k$-regular. It is well-known (see, e.g., \cite{GR}) that $G$ has eigenvalues $\pm k$ and $\pm\sqrt{q}$. Thus, $R(G)=\sqrt{q}$. Let $e_0$ be any edge of $G$ and $E_0=\{e_0\}$.
For each positive integer $t$, define the permutation assignment $\phi_t$ as in Theorem \ref{liftoneedge}.

The adjacency matrix of $G(\phi_2)$ can be written as $A(G(\phi_2))=A(G)\otimes I_2+B$, where $B$ is a matrix with only $\pm 2$ as nonzero eigenvalues.
Let $r$ be the number of vertices of $G(\phi_2)$. By the Courant-Weyl inequalities
$\lambda_{i+j-r}(A+B)\geqslant \lambda_i(A)+\lambda_j(B)$, we have $\lambda_{r/2-1}(A(G(\phi_2)))\geqslant \lambda_{r/2}(A(G)\otimes I_2)+\lambda_{r-1}(B)$ which gives $\lambda_{r/2-1}(G(\phi_2))\geqslant \sqrt{q}$.
In fact, since $G(\phi_2)$ has $\sqrt{q}$ as an eigenvalue with big multiplicity, one observes that
$\lambda_{r/2-1}(G(\phi_2))= \sqrt{q}$ and so $R(G)=\lambda_{r/2}(G(\phi_2))\leqslant \sqrt{q}$.

Let $e_0=\{v_0,v_1\}$. Let us consider the partition $\{v_0,v_1\}\cup W \cup W'$ of $V(G)$, where $W$ is the set of all vertices adjacent
to $v_0$ or $v_1$ and $W'$ is the set of vertices nonadjacent to $v_0,v_1$.
Define the vector $\mathbf{x}$ on $V(G)$ as
$$\mathbf{x}(w)=\left\{
\begin{array}{ll}
1 & w=v_0\ {\text or }\ w=v_1,\\
a & w\in W,\\
b & w\in W'.
\end{array}\right.$$
We now consider the vector $\mathbf{y}=(\mathbf{x},-\mathbf{x})$ on $V(G(\phi_2))$. Since the girth of $G$ is six, it is easy to see that $\mathbf{y}$ is an eigenvector of $G(\phi_2)$ with
the corresponding eigenvalue $\lambda$ if and only if $\lambda=qa-1,$  $a\lambda=qb+1$, and
$b\lambda=qb+a.$
Solving these equations in term of $\lambda$ gives $\lambda^3+(1-q)\lambda^2-3q\lambda+q^2-q=0$.
The value of the expression on the left side of this equation is 
$2$ and $\sqrt{q}-q$ for $\lambda=\sqrt{q-1}-1$ and $\lambda=\sqrt{q}-1$, respectively.
Therefore, there is a root between $\sqrt{q-1}-1$ and $\sqrt{q}-1$. This implies that $\sqrt{q-1}-1<R(G(\phi_2))<\sqrt{q}-1$. Finally, since $R(G)=\sqrt{q}$,
Theorem \ref{liftoneedge} implies that $\sqrt{q-1}-1<R(G(\phi_{2t}))=R(G(\phi_2))<\sqrt{q}-1$ for all $t$.
\end{proof}

The proof of Theorem \ref{infinit} can be used to obtain a slightly better bound on $R(G)$. 
Let $t=\sqrt{k-1}$. 
The value of $\lambda^3+(1-q)\lambda^2-3q\lambda+q^2-q$ is 
\begin{equation}\label{hteq}
(ht)^{-3}((2h^2-h^3)t^5+(h^3-2h^2)t^4+(4h^2-h)t^3+(3h-h^2)t^2-2ht-1)
\end{equation} for $\lambda=t-1-(ht)^{-1}$.  Note  that  (\ref{hteq})  is positive  for
$h=2$ and any $t$, whereas  it is negative for any $h>2$ if $t$ is large enough.  Therefore  we  find that 
$\sqrt{k-1}-1-\frac{1}{2\sqrt{k-1}}< R(G)<\sqrt{k-1}-1-\frac{1}{(2+\epsilon)\sqrt{k-1}}$ for every $\epsilon>0$ and any large $k$.

\end{document}